\documentclass[11pt]{smfart}
\usepackage{color}
\usepackage{amssymb,verbatim}
\usepackage{hyperref}
\usepackage{amsthm,array,amssymb,amscd,amsfonts,amssymb,latexsym, url}
\usepackage{amsmath}
\usepackage[all]{xy}
\usepackage[french]{babel}

\newcounter{spec}
{\end{list}}

\renewcommand{\P}{{\mathbf P}}

\newcommand{\Z}{{\mathbb Z}}
\newcommand{\Q}{{\mathbb Q}}
\newcommand{\C}{{\mathbb C}}

\renewcommand{\L}{{\mathbb L}}

\newcommand{\Hom}{{\operatorname{Hom}}}

\newcommand{\Spec}{{\operatorname{Spec \ }}}

\renewcommand{\lim}{\varprojlim}

\def\hfl#1#2#3{\smash{\mathop{\hbox to#3{\rightarrowfill}}\limits
^{\scriptstyle#1}_{\scriptstyle#2}}}

\def\gfl#1#2#3{\smash{\mathop{\hbox to#3{\leftarrowfill}}\limits
^{\scriptstyle#1}_{\scriptstyle#2}}}

\def\vfl#1#2#3{\llap{$\scriptstyle #1$}
\left\downarrow\vbox to#3{}\right.\rlap{$\scriptstyle #2$}}

\def\fleche(#1,#2)\dir(#3,#4)\long#5{%
\noalign{\leftput(#1,#2){\vector(#3,#4){#5}}}}

\def\ligne(#1,#2)\dir(#3,#4)\long#5{%
\noalign{\leftput(#1,#2){\lline(#3,#4){#5}}}}

\def\put(#1,#2)#3{\noalign{\setbox1=\hbox{%
    \kern #1\unitlength
    \raise #2\unitlength\hbox{$#3$}}%
    \ht1=0pt \wd1=0pt \dp1=0pt\box1}}
\numberwithin{equation}{section}

\newfont{\gothic}{eufb10}

\renewcommand{\qed}{{\hfill$\square$}}

\newtheorem{theo}{Th\'{e}or\`{e}me}[section]
\newtheorem{prop}[theo]{Proposition}

\newtheorem{lem}[theo]{Lemme}
\newtheorem{cor}[theo]{Corollaire}
\theoremstyle{definition}
\newtheorem{defi}[theo]{D\'efinition}
\theoremstyle{remark}
\newtheorem{rema}[theo]{Remarque}

\newcommand{\bthe}{\begin{theo}}
\newcommand{\ble}{\begin{lem}}
\newcommand{\bpr}{\begin{prop}}
\newcommand{\bco}{\begin{cor}}
\newcommand{\bde}{\begin{defi}}
\newcommand{\ethe}{\end{theo}}
\newcommand{\ele}{\end{lem}}
\newcommand{\epr}{\end{prop}}
\newcommand{\eco}{\end{cor}}
\newcommand{\ede}{\end{defi}}

\newcommand{\et}{{\operatorname{\acute{e}t}}}

\newcommand{\G}{{\mathbb G}}

\def\A{{\bf A}}
\def\C{{\bf C}}
 \def\Q{{\bf Q}}
 \def\Z{{\bf Z}}
\def\G{{\bf G}}
\def\P{{\bf P}}
\def\aA{{\mathbb A}}

\DeclareFontFamily{U}{wncy}{}
\DeclareFontShape{U}{wncy}{m}{n}{%
<5>wncyr5%
<6>wncyr6%
<7>wncyr7%
<8>wncyr8%
<9>wncyr9%
<10>wncyr10%
<11>wncyr10%
<12>wncyr6%
<14>wncyr7%
<17>wncyr8%
<20>wncyr10%
<25>wncyr10}{}
\DeclareMathAlphabet{\cyr}{U}{wncy}{m}{n}
\begin{document}

  \title[ ] {Approximation forte pour les espaces homog\`enes de groupes semi-simples  sur le corps des fonctions d'une courbe  alg\'ebrique complexe}

\author{J.-L. Colliot-Th\'el\`ene}
\address{Universit\'e Paris Sud\\Math\'ematiques, B\^atiment 425\\91405 Orsay Cedex\\France}
\email{jlct@math.u-psud.fr }

\date{15  avril 2016}
\maketitle

 \begin{abstract}
Sur un corps de fonctions d'une variable sur le corps des complexes,
 l'approximation forte hors d'un ensemble fini non vide de places vaut pour pour
 tout espace homog\`ene d'un groupe semi-simple simplement connexe.
  En particulier l'approximation forte hors d'un ensemble fini non vide de places 
 vaut pour les quadriques affines lisses de dimension au moins 2 sur un tel corps.
 L'approximation forte hors  d'un ensemble fini de places ne vaut pas pour le groupe multiplicatif.
   \end{abstract}
   
   \begin{altabstract} Let $K$ be the function field of a curve over the complex field.
Let $X$ be a homogeneous space of  a semisimple linear algebraic group  over $K$.
   Strong approximation holds for $X$  outside any  finite nonempty set of places of $K$.
 Strong approximation fails for tori over $K$.
   \end{altabstract}

\begin{center}

\end{center}

\section{Introduction}

Sur un corps de nombres,  l'approximation faible, resp.  l'approximation forte, 
ont  \'et\'e depuis longtemps discut\'ees
pour les groupes lin\'eaires connexes, leurs espaces homog\`enes, et
certaines vari\'et\'es rationnellement connexes. 
Cette \'etude va de pair avec celle du principe local-global pour l'existence
de points rationnels, resp. de points entiers.
Sur un corps de nombres, on sait  qu'une condition n\'ecessaire
pour  qu'une vari\'et\'e lisse satisfasse l'approximation forte est qu'elle soit
g\'eom\'etriquement simplement connexe (Kneser \cite{K} pour les groupes,
 Minchev \cite{M} en g\'en\'eral).

Depuis une dizaine d'ann\'ees, les techniques de d\'eformation de courbes en
g\'eom\'etrie alg\'ebrique complexe sont appliqu\'ees  \`a l'\'etude
de ces questions pour les vari\'et\'es rationnellement connexes
d\'efinies sur un corps $K=\C(\Gamma)$  
de fonctions d'une variable sur les
complexes,  situation a priori plus simple que celle des  vari\'et\'es sur les  corps de nombres.
Dans l'article r\'ecent \cite{ChZh}, Chen et Zhu \'etablissent l'approximation forte sur $K$
pour le compl\'ementaire d'hypersurfaces lisses dans les intersections compl\`etes lisses
dans un espace projectif, sous certaines conditions sur les  multidegr\'es.

Un cas tr\`es  particulier   de leurs r\'esultats
est celui  des quadriques affines d'\'equation
$\sum_{i=1}^n a_{i} x_{i}^2=b,$
avec les $a_{i}$ et $b$ dans $K^*$ et $n \geq 4$, pour lesquelles ils 
\'etablissent l'approximation forte, et en particulier un principe local-global pour les points entiers, analogues d'un
r\'esultat de M. Kneser sur les corps de nombres.

Pour les espaces homog\`enes  de groupes lin\'eaires connexes sur $K=\C(\Gamma)$, avec P. Gille nous avons 
\'etabli l'approximation faible \cite{CTG}. Un th\'eor\`eme de Harder \cite{H}
donne l'approximation forte (en dehors d'un ensemble fini non vide de places)
pour les groupes semi-simples simplement connexes. La question ne semble pas
avoir \'et\'e \'etudi\'ee pour les autres groupes 
lin\'eaires et  leurs espaces homog\`enes.

Dans cette note, je remarque  que pour $K=\C(\Gamma)$
 l'approximation forte  hors d'un ensemble fini non vide de places
vaut pour tout  $K$-groupe semisimple et plus g\'en\'eralement pour tout 
espace homog\`ene d'un $K$-groupe semi-simple (Th\'eor\`eme \ref{appfortesemi-simple}).
La d\'emonstration combine le th\'eor\`eme de Harder avec un th\'eor\`eme (Th\'eor\`eme \ref{surjH1}) 
dont la d\'emonstration \'etait d\'ej\`a essentiellement dans \cite{CTG}, et qui repose sur
le th\'eor\`eme d'existence de Riemann.
 
Dans le cas particulier des quadriques affines mentionn\'e ci-dessus, ceci donne l'approximation forte
et le principe local-global pour les points entiers pour $n \geq 3$ (Corollaire \ref{troisvariables}, dont 
la d\'emonstration n'utilise que le th\'eor\`eme de Harder).

 On voit  par contre facilement que l'approximation forte  
  est en d\'efaut pour le groupe multiplicatif $\G_{m,K}$.

\section{L'approximation forte, le cas de l'espace affine}

Soit $T$ un sch\'ema de Dedekind int\`egre, de corps des fractions $K$.
Pour $P$ point ferm\'e de $T$ on note $R_{P}$ le compl\'et\'e de l'anneau local de $T$
en $P$
et $K_{P}$ le corps des fractions de $R_{P}$. On note ${\aA}_{K}$ l'anneau des ad\`eles de $K$. 
Pour $S \subset T$ un ensemble fini, on note ${\aA}^S_{K}$ l'anneau des ad\`eles
hors de $S$.
Soit $X$ une $K$-vari\'et\'e lisse. 
Pour tout point ferm\'e $P \in T$, on munit $X(K_{P})$ de la topologie d\'efinie
par celle de $K_{P}$.
On note $X(\aA_{K})$ l'espace topologique des points ad\`eliques de $X$,
et pour $S \subset T$ un ensemble fini, 
$X(\aA^S_{K})$ l'espace des points ad\'eliques hors de $S$.

Pour $S \subset T$ un ensemble fini,
on dit que la $K$-vari\'et\'e  $X$ satisfait l'approximation forte sur $T$ hors de $S$
si l'image diagonale de $X(K)$ est dense dans la projection sur $X(\aA^S_{K})$ de
$X(\aA_{K})$. En d'autres termes, si $X(\aA_{K})$ est non vide, alors $X(K$) est non vide
et son image diagonale dans $X(\aA^S_{K})$ est dense.
Si la propri\'et\'e vaut pour $S=\emptyset$, on dit  que $X$
satisfait l'approximation forte sur $T$.

Si la $K$-vari\'et\'e $X$ satisfait l'approximation forte hors de $S$,
alors elle satisfait l'approximation forte hors de tout ensemble fini 
$S'$ contenant $S$.

\begin{lem}\label{produit}
Soient $X$ et $Y$ deux $K$-vari\'et\'es lisses g\'eom\'etriquement  int\`egres.
si l'approximation forte hors d'un ensemble fini $S\subset T$  vaut pour  la $K$-vari\'et\'e $X$
et la $K$-vari\'et\'e $Y$, alors elle vaut  hors de $S \subset T$ pour la $K$-vari\'et\'e $X\times_{K}Y$.
\end{lem}
\begin{proof}
Ceci r\'esulte simplement du fait suivant :
Pour $P$ point ferm\'e de $T$, tout voisinage ouvert d'un $K_{P}$-point $(M,N) \in (X\times_{K}Y)(K_{P})$
contient un voisinage ouvert  de $(M,N)$  produit d'un voisinage ouvert de $M \in X(K_{P})$ par
un voisinage ouvert de $N\in Y(K_{P})$. 
\end{proof}

On a aussi une version ``tordue'' du lemme pr\'ec\'edent.

 \begin{lem}\label{Weil}
Soit $T_{1}/T$ un rev\^{e}tement fini de sch\'emas de Dedekind int\`egres  tel que l'extension 
correspondante de corps des fractions $L/K$ soit s\'eparable. Soit $X$ une $L$-vari\'et\'e.
Soit $S \subset T$ un ensemble fini et $S_{1} \subset T_{1}$ son image r\'eciproque.
Si  une $L$-vari\'et\'e lisse g\'eom\'etriquement int\`egre $X$ satisfait l'approximation forte sur $T_{1}$
hors de $S_{1}$
alors la $K$-vari\'et\'e descendue \`a la Weil $R_{L/K}X$ satisfait l'approximation forte sur $T$
hors de $S$.
\end{lem}

\section{L'approximation forte pour les groupes semi-simples et leurs espaces homog\`enes}

On a le th\'eor\`eme d'approximation forte g\'en\'eral suivant, d\^{u} \`a P. Gille  \cite[Corollaire 5.11]{G}.

\begin{theo}    \label{generaldedekind}  Soient $A$ un anneau de Dedekind, $K$ son corps des fractions
et $G$ un $K$-groupe semi-simple simplement connexe absolument  presque $K$-simple 
$K$-isotrope. Supposons la $K$-vari\'et\'e $G$ $K$-rationnelle, ou plus g\'en\'eralement  r\'etracte rationnelle.
Alors l'approximation forte  sur $\Spec A$ vaut pour $G$.
\end{theo}

 En utilisant le lemme \ref{Weil}, on voit que la conclusion vaut
  aussi pour toute restriction \`a la Weil d'un tel groupe d\'efini sur une
 extension finie de $K$, et ensuite pour tout
 produit sur $K$ de tels groupes.
 
Tout  $K$-groupe semi-simple simplement connexe quasi-d\'eploy\'e est produit
de groupes du type ci-dessus : le  tore maximal d'un sous-groupe de Borel
 d'un $K$-groupe semi-simple simplement connexe  quasi-d\'eploy\'e est un $K$-tore quasi-trivial,
 et tout groupe  simple quasi-d\'eploy\'e est  isotrope \cite[Thm. 1.1]{T}).
 Le th\'eor\`eme \ref{generaldedekind} a donc comme cons\'equence 
l'\'enonc\'e suivant, d\^{u}  \`a Harder \cite[Satz 2.2.1]{H}, et qui suffit aux besoins du pr\'esent article.

\begin{theo}\label{harder}  Soient $A$ un anneau de Dedekind, $K$ son corps des fractions
et $G$ un $K$-groupe semi-simple simplement connexe quasi-d\'eploy\'e.
Alors l'approximation forte  sur $\Spec A$ vaut pour $G$.
\end{theo}

 \begin{cor}\label{simplementconnexe}
  Soit $\Gamma$ une courbe connexe, projective et lisse sur $\C$,  soit $K=\C(\Gamma)$.
 L'approximation forte hors de tout ensemble fini non vide $S \subset \Gamma(\C)$
 vaut pour tout $K$-groupe semi-simple simplement connexe $G$.
 \end{cor}
 
 \begin{proof}
Le corps  $K=\C(\Gamma)$ \'etant un corps $C_{1}$ (Tsen),   tout $K$-groupe $G$ est quasi-d\'eploy\'e. 
(Springer \cite[Prop. 1.1, Prop. 1.2]{Sp}; voir \cite[III.2.3]{S} pour le cas g\'en\'eral de la conjecture I
de Serre, \'etabli par Steinberg).
 Le th\'eor\`eme \ref{harder} s'applique donc
 \`a l'ouvert affine $\Gamma \setminus  S$, qui est le spectre d'un anneau de Dedekind,
 et \'etablit le r\'esultat pour $G$.
\end{proof}

 \begin{prop}\label{PoitouTate}
 (i) Soit $F=\C((t))$. Pour tout module galoisien fini $\mu$ sur $F$, de dual le module galoisien fini  $\hat{\mu}= \Hom_{\Z}({\mu}, \Q/\Z(1))$,
 l'accouplement \'evident
 $$H^1(F,\mu) \times \hat{\mu}(F) \to H^1(F,\Q/\Z(1))=\Q/\Z$$
 est une dualit\'e parfaite de groupes finis.
 
 (ii)
 Soit $\Gamma$ une courbe connexe, projective et lisse sur $\C$,  et soit $K=\C(\Gamma)$.
 Pour tout module galoisien fini $\mu$ sur $F$, de dual $\hat{\mu}= \Hom_{\Z}({\mu}, \Q/\Z(1))$,
 on a une suite exacte naturelle
 $$ H^1(K,\mu) \to \bigoplus_{P \in \Gamma(\C)} H^1(K_{P},\mu) \to \Hom(\hat{\mu}(K), \Q/\Z) \to 0,$$
chaque application  $H^1(K_{P},\mu) \to \Hom(\hat{\mu}(K), \Q/\Z)$ \'etant donn\'ee par la fl\`eche compos\'ee
 $H^1(K_{P},\mu) \to \Hom(\hat{\mu}(K_{P}), \Q/\Z) \to \Hom(\hat{\mu}(K), \Q/\Z)$ et \'etant en particulier
 surjective.
 
 (iii) Soit $S \subset \Gamma(\C)$  un ensemble non vide, et soit $\mu$ comme ci-dessus.
 L'application diagonale  $$ H^1(K,\mu) \to \bigoplus_{P \notin S} H^1(K_{P},\mu)$$
 est surjective.
 
 \end{prop}
  
\begin{proof}  L'\'enonc\'e (i) r\'esulte simplement du fait que le groupe de Galois absolu
de $\C((t))$ est $\hat{\Z}$ (Puiseux),  engendr\'e topologiquement par un \'el\'ement $\sigma$, 
donc que $H^1(F,\mu) = \mu/(1-\sigma)\mu$.

Pour   $\mu=\mu_{n} $  le groupe constant d\'efini par les racines $n$-i\`emes de l'unit\'e, l'\'enonc\'e (ii)
r\'esulte,  via la suite de Kummer, du  fait que le groupe des $\C$-points de la jacobienne de $\Gamma$
est divisible.

Pour $\mu$ quelconque, l'\'enonc\'e (ii) 
est  un th\'eor\`eme de type Poitou-Tate qui ne semble pas avoir \'et\'e explicit\'e dans la litt\'erature classique.
C'est maintenant  le cas le plus simple  ($d=-1$) d'un th\'eor\`eme d'Izquierdo  sur les courbes sur les corps
$d$-locaux (Izquierdo  \cite[Thm. 2.7]{Izq}), \'etendant des r\'esultats  sur les cas $d=0 $ \cite{CTH} et $d=1$   \cite[\S 2]{HSS}.
Nous laissons au lecteur le soin d'adapter les arguments de ces articles au cas plus simple ici consid\'er\'e.

L'\'enonc\'e (iii) r\'esulte de l'\'enonc\'e (ii).
\end{proof}

L'\'enonc\'e (iii) est un cas particulier de l'\'enonc\'e suivant, qui seul sera utilis\'e pour la d\'emonstration
du th\'eor\`eme \ref{appfortesemi-simple}.

 \begin{theo}\label{surjH1}
 Soit   $\Gamma$ une courbe  connexe, projective et lisse sur $\C$, et soit $K=\C(\Gamma)$.
 Soit  $S \subset \Gamma(\C)$ fini  non vide.
 Pour tout $K$-groupe lin\'eaire $G$, l'application diagonale
 $$ H^1(K,G) \to \oplus_{P \notin S} H^1(K_{P},G)$$
 est surjective. 
  \end{theo}
  \begin{proof}
 La somme directe d\'esigne ici le sous-ensemble du produit des ensembles point\'es $H^1(K_{P},G)$
  form\'e des \'el\'ements $\{\xi_{P}\}$
avec $\xi_{P}=1$ pour presque tout point $P$.
Soit $\{\xi_{P}\} \in  \oplus_{P \notin S} H^1(K_{P},G)$.
  Il suffit de suivre la d\'emonstration du th\'eor\`eme 4.2 de \cite{CTG}.
On commence par se r\'eduire au cas o\`u $G$ est un $K$-groupe fini.
On peut supposer que $S$ est r\'eduit \`a un seul point $P_{0}$.
On choisit ensuite un ensemble fini $T \subset \Gamma(\C)$, contenant $P_{0}$,
de compl\'ementaire $U \subset \Gamma$
 tel que  $G/K$ s'\'etende en un $U$-sch\'ema en groupes fini \'etale $ \mathcal{G}$ sur $U$,
 et que $\xi_{P}=1$ pour $P\notin U$.
  On consid\`ere ensuite l'extension $L/K$  maximale non ramifi\'ee en dehors
 de $T$. 
 La d\'emonstration de \cite[Thm. 4.2]{CTG}, qui repose sur le th\'eor\`eme d'existence de Riemann,
   produit $\eta  \in H^1(L/K,G(L)) \subset H^1(K,G)$ d'image
 $\xi_{P}$ pour $P\in T \setminus \{P_{0}\}$. Pour $P \notin T$, l'image de $\eta$
 est  $1 \in  H^1(K_{P},G)$, puisque $\eta$ est dans l'image de $H^1_{\et}(U, \mathcal{G})$.
   \end{proof}

\begin{theo} \label{appfortesemi-simple}
Soit $\Gamma$ une courbe  connexe, projective et lisse sur $\C$, et soit $K=\C(\Gamma)$.
L'approximation forte hors de tout ensemble fini non vide $S \subset \Gamma(\C)$
 vaut pour tout espace homog\`ene d'un $K$-groupe semi-simple $G$.
\end{theo}

\begin{proof} Soit $X$ un tel espace homog\`ene. 
Comme on n'a fait aucune hypoth\`ese
sur les stabilisateurs g\'eom\'etriques, on peut supposer $G$ simplement connexe.
On sait que l'on a $X(K)\neq \emptyset$. On a donc $X=G/H$ pour $H \subset G$
un $K$-sous-groupe ferm\'e. Comme $G$ est lin\'eaire connexe, les ensembles $H^1(K,G)$ et $H^1(K_{P},G)$ sont r\'eduits
chacun \`a un point. On a des suites exactes d'ensembles point\'es (\cite[I.5.4, Prop. 36  et Cor. 1]{S}) :

 $$
\begin{matrix}
G(K)                &    \hfl{}{}{4mm}  &                    X(K)                          &   \hfl{}{}{4mm}   & H^1(K,H)                                                                 &\hfl{}{}{4mm} &1 \cr
\vfl{}{}{4mm} &              &            \vfl{}{}{4mm}                 &            &       \vfl{}{}{4mm}                                                       &  &                                                       \cr
G(\aA_{K}^{P_{0}}) &  \hfl{}{}{4mm}   &         X(\aA_{K}^{P_{0}})        &\hfl{}{}{4mm}     & \oplus_{P \neq P_{0}}H^1(K_{P},H)                   &\hfl{}{}{4mm} &  1 . \cr
\end{matrix}
 $$
Le  th\'eor\`eme \ref{surjH1}, une chasse au diagramme et le corollaire  \ref{simplementconnexe} (approximation forte pour $G$
semi-simple simplement connexe)
donnent alors l'approximation forte pour $X$ hors de la place $P_{0}$.
\end{proof}

\begin{rema}\label{ab}
 Le th\'eor\`eme \ref{simplementconnexe} et  la proposition \ref{PoitouTate} suffisent  \`a \'etablir 
  le th\'eor\`eme \ref{appfortesemi-simple} pour les espaces homog\`enes $X=G/H$ dans les
  cas suivants :

   (i) $X=G$ avec $G$ semi-simple, et plus g\'en\'eralement :
  
  (ii) $X=G/A$ avec $G$ semi-simple simplement connexe et $A$ sous-$K$-groupe ab\'elien fini.
  
  (iii) $G$ semi-simple et $H$ connexe. 
\end{rema}

  \begin{rema}
 Sur un corps de nombres, un groupe non simplement connexe ne saurait satisfaire l'approximation forte hors d'un nombre fini de places
 (Kneser \cite{K}).  Plus g\'en\'eralement, Minchev \cite{M} 
    a montr\'e que si une vari\'et\'e normale sur un corps de nombres satisfait l'approximation forte
 hors d'un nombre fini de places, alors elle est g\'eom\'etriquement simplement connexe.
 La situation sur $K=\C(\Gamma)$ 
 est diff\'erente, comme l'ont d\'ej\`a remarqu\'e  Chen et Zhu \cite[Cor. 1.5 et  \S 5]{ChZh}.
 \end{rema}

Pour $n\geq 4$, l'\'enonc\'e suivant  est un cas particulier
d'un r\'esultat de Chen et Zhu  \cite[Thm. 1.4]{ChZh}.

\begin{cor}\label{troisvariables}
Soit $\Gamma$ une courbe  connexe, projective et lisse sur $\C$, soit $K=\C(\Gamma)$ et soit $S \subset \Gamma(\C)$ fini non vide.
  Soient $n\geq 3$ et $q(x_{1},\dots,x_{n})$ une forme quadratique non d\'eg\'en\'er\'ee
sur $K$. Soit $b \in K^*$. L'approximation forte hors de $S$ vaut pour la quadrique affine 
d'\'equation $q(x_{1},\dots,x_{n})=b.$
\end{cor}
 \begin{proof}
Comme il est bien connu  (cf. \cite[\S 5]{CTX}),
 la quadrique affine sur $K=\C(t)$ donn\'ee par l'\'equation
 $q(x_{1},\dots,x_{n})=b$
pour $n \geq 2$ poss\`ede un $K$-point et  pour $n\geq 3$  peut s'\'ecrire, apr\`es choix d'un $K$-point, comme
 quotient du groupe des spineurs de la forme quadratique $q$
 par un $K$-groupe lin\'eaire connexe (tore pour $n=3$, groupe semi-simple simplement connexe pour $n \geq 4$).
 L'\'enonc\'e est alors un cas particulier du th\'eor\`eme  
 \ref{appfortesemi-simple} -- pour $G$ semi-simple simplement connexe et $H$ connexe, si bien que
 le r\'esultat ici est une cons\'equence directe du th\'eor\`eme de Harder. 
   \end{proof}

\begin{cor}\label{quadaff} Soient $n \geq 3$ et   $b_{1}(t), \dots, b_{n}(t) \in \C[t]$ des \'el\'ements non nuls.
 Pour $b(t)\in \C[t]$,  l'\'equation
 $$ \sum_{i=1}^n a_{i}(t) x_{i}(t)^2=b(t)$$
 a des solutions avec $x_{i}(t) \in \C[t]$ si et seulement si elle a des solutions
 modulo toute puissance de $\prod_{i=1}^na_{i}(t)$. 
 \end{cor}

\begin{rema}
Sur un corps de nombres totalement imaginaire, pour une quadrique affine d'\'equation
 $ q(x_{1}, \dots, x_{n})=b,$ avec $q$ forme quadratique non d\'eg\'en\'er\'ee et $b \in k^*$,
 l'approximation forte vaut 
si $n \geq 4$, mais ne vaut pas en g\'en\'eral pour $n=3$ (cf. \cite{CTX}).
\end{rema}

\begin{theo}\label{descente}
Soit $\Gamma$ une courbe  connexe, projective et lisse sur $\C$, soit $K=\C(\Gamma)$ et soit $P_{0} \in \Gamma(\C)$.
Soit $G$ un $K$-groupe lin\'eaire. 
Soit $X$ une $K$-vari\'et\'e lisse g\'eom\'etriquement int\`egre et $Y \to X$ un torseur sur $X$
sous le $K$-groupe $G$. Si pour tout  1-cocycle $\zeta \in Z^1(K,G)$, l'espace total du torseur tordu 
 $Y^{\zeta} $ satisfait
l'approximation forte hors de $P_{0} $ et satisfait $Y^{\zeta}(K_{P_{0}}) \neq \emptyset$, alors
 $X$ satisfait l'approximation forte hors de $S$.
\end{theo}

\begin{proof}
Ceci se d\'eduit du th\'eor\`eme \ref{surjH1} par le m\^{e}me argument que pour
le th\'eor\`eme  
\ref{appfortesemi-simple}.
\end{proof}

Pour $n \geq 3$, le corollaire suivant est un cas particulier de \cite[Cor. 1.5]{ChZh}.

\begin{cor}\label{descentequadrique}
Soit $\Gamma$ une courbe  connexe, projective et lisse sur $\C$,  et soit $K=\C(\Gamma)$.
Soient $n\geq 2$ un entier et $q(x_{0},\dots,x_{n})$ une forme quadratique non d\'eg\'en\'er\'ee sur $K$.
Soit  $X$ le compl\'ementaire de la quadrique $q=0$ dans $\P^{n}$.
Pour tout ensemble fini non vide $S \subset \Gamma(\C)$, 
l'approximation forte 
  hors de $S$ vaut pour $X$.
\end{cor}
\begin{proof} 
Soit $Y\subset \A^{n+1}_{K}$ la quadrique d'\'equation $q(x_{0},\dots,x_{n})=1$.
La projection $Y \to X$ fait de $Y$ un torseur sur $X$ sous le groupe $\mu_{2}$.
On a $H^1(K,\mu_{2})=K^*/K^{*2}$, tout torseur tordu de $Y \to X$
a son espace total d\'efini par une \'equation 
 $q(x_{0},\dots,x_{n})=c$ pour $c \in K^*$ convenable.
 Le corollaire Ê\ref{troisvariables} et le th\'eor\`eme \ref{descente} donnent le r\'esultat.
 \end{proof}

\begin{rema}
Dans le cas  $\mu=\mu_{d}$, le groupe des
racines $d$-i\`emes de l'unit\'e, l'argument donn\'e au th\'eor\`eme \ref{descente}  est tr\`es proche de celui
utilis\'e dans  \cite[\S 5]{ChZh} pour passer   de l'approximation forte
hors de $P_{0}$ sur les hypersurfaces affines lisses d'\'equation $f(x_{0}, \dots, x_{n})=c$,
avec $c\neq 0$ et $f$ homog\`ene de degr\'e $d$ non singuli\`ere,
(approximation forte \'etablie dans  \cite{ChZh}  pour $d^2 \leq n+1$)
\`a l'approximation forte pour les ouverts de $\P^n_{K}$
d'\'equation $f(x_{0}, \dots, x_{n})\neq 0$. 
\end{rema}

\section{L'approximation forte ne vaut pas pour les tores}

\begin{prop}
 Pour $K=\C(\P^1)=\C(t)$,  pour tout ensemble fini $S$ de points de $\P^1(\C)$,
 l'approximation forte hors de $S$ est en d\'efaut pour la $K$-vari\'et\'e  $\G_{m,K}$.
 \end{prop}
 \begin{proof}
 Pour $P \in \P^1(\C)$,
 notons 
 $U_{P} \subset  K_{P}^*$
 le groupe des unit\'es.
 On peut supposer que $S$ consiste en les points $t=e_{1}, \dots, e_{n}  , \infty$. 
Le groupe $M$  intersection de   $K^*$ et de  $\prod_{P \notin S} U_{P}$ est form\'e des \'el\'ements
de la forme $c. \prod_{i=1}^n (t-e_{i})^{n_{i}}$ avec $c \in \C^*$ et les $n_{i} \in \Z$.
  Fixons $a,b \in \C$ distincts et distincts des $e_{i}$. Si $K^*$ \'etait dense dans
$\prod_{P \notin S} U_{P} $
l'image de $M$
par l'\'evaluation simultan\'ee en $a$ et $b$ serait \'egale \`a $\C^* \times \C^*$.
Prenant le quotient sur les deux facteurs, ceci  impliquerait que tout \'el\'ement de $\C^*$
appartient au sous-groupe engendr\'e par les $(a-e_{i})/(b-e_{i})$, $i=1,\dots, n$, donc que $\C^*$
serait un groupe de type fini.
 \end{proof}

 \medskip
 
La proposition montre que l'approximation forte sur $\C[t]$ ne vaut pas pour une \'equation
$xy= 1$
sur $K=\C(t)$. La condition $n \geq 3$ dans le corollaire \ref{quadaff} est  donc n\'ecessaire.

\medskip

On peut aussi donner un
 un contre-exemple au principe local-global
 pour les solutions dans $\C[t]$ d'une \'equation
 $$a(t)x^2+b(t)y^2=c(t)$$
 avec $a(t), b(t), c(t) \in \C[t]$ non nuls.
 L'exemple suivant m'a \'et\'e communiqu\'e par D. Izquierdo :
$$ (t+1)x^2 + t^2y^2=1.$$
 Le lemme de Hensel donne des solutions dans tous les compl\'et\'es
 de $\C[t]$ en les points de $\A^1(\C)$, et  il y a une solution dans $\C(t)$
 par le th\'eor\`eme de Tsen. Mais il n'y a pas de solution  avec $x,y \in \C[t]$. C'est clair si
 $x=0$, et si $x \neq 0$
 le polyn\^{o}me $(t+1)x^2$ est de degr\'e impair  et le polyn\^{o}me  $t^2y^2$ est  de degr\'e pair, leur somme
 ne peut \^etre une constante.

\begin{rema}
Dans \cite[\S 3]{CTG} nous avons donn\'e un contre-exemple \`a l'approximation faible pour une surface d'Enriques $X/\C(t)$.
L'obstruction utilis\'ee est une obstruction de r\'eciprocit\'e associ\'ee \`a un \'el\'ement de $H^1_{\et}(X,\Z/2)$.
Comme $X$ est projective, cela donne imm\'ediatement un contre-exemple \`a l'approximation forte
en  dehors d'ensembles finis de places de $K$. 

La situation est l\`a tr\`es diff\'erente de celle
d'une isog\'enie de groupes lin\'eaires connexes, comme discut\'ee dans la d\'emonstration du th\'eor\`eme \ref{appfortesemi-simple}.
Dans \cite[\S 3]{CTG}, on a un $\Z/2$-torseur $Y \to X$,
d\'efinissant pour tout point $P$ une application $X(K_{P})\to H^1(K_{P},\Z/2)$.
Pour presque tout point $P$, cette application a son image r\'eduite \`a z\'ero.
Si par contre on consid\`ere une isog\'enie
$$ 1 \to \mu \to \tilde{G} \to G \to 1$$
avec $\tilde{G}$ et $G$ des $K$-groupes lin\'eaires connexes,
pour tout point $P$ l'application associ\'ee $G(K_{P}) \to H^1(K_{P},\mu)$
est surjective.
\end{rema}

\end{document}